\documentclass[12pt]{article}
\usepackage{epic}
\usepackage{authblk}
\usepackage{amssymb}
\usepackage{graphicx}
\begin{document}
\newcommand{\ssim}{\mathrel{\raisebox{0.3ex}{$\scriptstyle\sim$}}}
\newcommand{\nssim}
{\mathrel{\raisebox{0.3ex}{$\scriptstyle\nsim$}}}

\newcommand{\up}{\vspace*{-0.136cm}}
\newcommand{\qed}{\hfill$\rule{.05in}{.1in}$ \vspace{.3cm}}
\newcommand{\pf}{\noindent{\bf Proof: }}
\newtheorem{thm}{Theorem}
\newtheorem{lem}{Lemma}
\newtheorem{prop}{Proposition}
\newtheorem{ex}{Example}
\newtheorem{cor}{Corollary}
\newtheorem{conj}{Conjecture}
\newtheorem{prob}{Problem}
\newtheorem{claim}{Claim}
\newcommand{\beq}{\begin{equation}}
\newcommand{\eeq}{\end{equation}}
\newcommand{\<}[1]{\left\langle{#1}\right\rangle}
\newcommand{\be}{\begin{enumerate}}
\newcommand{\ee}{\end{enumerate}}
\newcommand{\al}{\alpha}
\newcommand{\ep}{\epsilon}
\newcommand{\si}{\sigma}
\newcommand{\om}{\omega}
\newcommand{\la}{\lambda}
\newcommand{\La}{\Lambda}
\newcommand{\ga}{\gamma}
\newcommand{\im}{\Rightarrow}
\newcommand{\2}{\vspace{.2cm}}
\newcommand{\es}{\emptyset}
\newcommand{\ad}{\ssim}
\newcommand{\nad}{\pm}
\newcommand{\lz}{\langle}
\newcommand{\rz}{\rangle}
\newcommand{\bh}{\hat}

\title{On Camby--Plein's Characterization of Domination Perfect Graphs}
\author{Vadim Zverovich\footnote{E-mail: vadim.zverovich@uwe.ac.uk}}
\affil{\footnotesize Mathematics and Statistics Research Group, University of the West of England, Bristol, UK}

\maketitle

\vspace*{-.3cm}
\begin{abstract}
We show that all results stated in [E.~Camby, F.~Plein, \textit{Discrete Appl. Math.}~{\bf 217} (2017) 711--717] are either previously known or incorrect.
For example, Camby and Plein claimed to provide counterexamples to the 1995 characterization of domination perfect graphs due to Zverovich and Zverovich; however, these counterexamples are not valid. 
Moreover, the new characterization of domination perfect graphs proposed in that paper is incorrect.

For completeness, we present a relatively brief proof of the 1995 characterization of domination perfect graphs due to Zverovich and Zverovich.

\vspace*{.3cm} {\footnotesize \noindent Keywords: {\it
domination perfect graphs, domination number, independent domination
number, induced subgraph characterization
}}
\end{abstract}

\section{Summary of Results on Domination Perfect\\ Graphs}

All graphs will be  finite  and  undirected,  without  loops  or
multiple edges. If $G$ is a graph, $V(G)$ denotes the set,  and $\vert G\vert $
the number, of vertices in $G$. We will denote the neighborhood of
a vertex $x$ by $N(x)$. More generally, $N(X) = \bigcup^{}_{x\in X}N(x)$ for $X\subseteq V(G)$.
We will write $x\ssim X$ ($x\nssim X$)  to  indicate  that  a  vertex $x$  is
adjacent to all vertices (no vertex) of $X\subseteq V(G)$.  For  a  set  of
vertices $X$, $\langle X\rangle$ denotes the subgraph of $G$ induced by $X$.
If $X$, $Y$ are subsets  of $V(G)$,  then $X$  {\it dominates} $Y$  if
$Y\subseteq N(X)\cup X$. In particular, if $X$ dominates $V(G)$, then $X$ is called a
{\it dominating set} of $G$. An {\it independent dominating set} is  a  vertex
subset that is both independent and dominating, or equivalently,
is maximal  independent.  The  {\it domination  number} $\gamma (G)$  is  the
minimum cardinality taken over all dominating sets of $G$, and the
{\it independent domination number} $i (G)$ is  the  minimum  cardinality
taken over all maximal independent sets of vertices of $G$.
In all our figures a dotted line will always mean that the
corresponding edge may or may not be  a  part  of  the  depicted
graph.

Sumner and Moore \cite{Sum1} define a graph $G$  to  be  {\it domination
perfect} if $\gamma (H) = i (H)$, for every induced subgraph $H$ of $G$.
A graph $G$ is called {\it minimal  domination  imperfect}  if $G$  is  not
domination perfect and $\gamma (H) = i (H)$,  for  every  proper  induced
subgraph $H$ of $G$.
The related classes of graphs such as irredundance perfect graphs,
upper domination perfect graphs and upper irredundance 
perfect graphs are studied as well (e.g., see \cite{Gut,Vol1,Vol2,Vol3,Zve2}).
Various real-life applications of domination in graphs can be found in \cite{Zve3}.

Allan and Laskar [1]  established  that $K_{1,3}$-free  graphs are  domination perfect:

\begin{thm}
[Allan and Laskar \cite{All}]  
\label{A-L}
If $G$  has  no  induced  subgraph
isomorphic to $K_{1,3}$, then $\gamma (G) = i (G)$.
\end{thm}

This theorem implies that line graphs and middle graphs are domination perfect (the former proved independently by Gupta; see Theorem~10.5 in \cite{Har1}).
Theorem  \ref{A-L} also
improves the result of Mitchell and Hedetniemi \cite{Mit} that the
line graph $L(T)$  of  a  tree $T$  satisfies $\gamma (L(T))=i(L(T))$.
We may also mention the characterization of trees $T$ with $\gamma (T) = i (T)$  
due to Harary and Livingston \cite{Har2}.

The following results were obtained by Sumner and Moore \cite{Sum1,Sum2}:

\begin{thm}
[Sumner and Moore \cite{Sum1,Sum2}] 
A graph $G$  is  domination  perfect
if and only if $\gamma (H) = i (H)$  for every induced subgraph $H$ of $G$ with
$\gamma (H)=2$.
\end{thm}

\begin{thm}
[Sumner and Moore \cite{Sum1,Sum2}] 
\label{Chordal}
If $G$  is  chordal,  then $G$  is
domination perfect if and only if $G$ does not contain  an  induced  subgraph isomorphic to the graph $G_{1}$ in Figure \ref{17graphs}.
\end{thm}

Let 
$
{\cal A} = \{H\,:\, \vert H\vert \le 8, \;\gamma (H)=2,\; i(H)>2\}.
$

\begin{thm}
[Sumner and Moore \cite{Sum1,Sum2}] 
\label{Th4}
If $G$ contains no member of 
${\cal A}$ as an induced subgraph and also contains no induced copy of the graph $S$ in Figure \ref{Sgraph}, 
then $G$ is domination perfect.
\end{thm}

Note that the list of forbidden subgraphs for this theorem in \cite{Sum2} 
contains a superfluous graph, namely the graph $S$ minus the edge connecting a vertex of degree 5 to a vertex of degree 4.

\vspace*{0.2cm}
\begin{figure}[h!] \centering\includegraphics[width=0.32\linewidth]{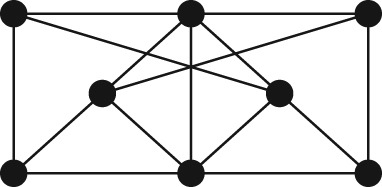} 
\caption{Graph $S$.}
\label{Sgraph} 
\end{figure}

\newpage

In addition to   chordal   graphs,   Sumner  and   Moore  \cite{Sum1,Sum2}
characterized planar domination perfect graphs:

\begin{thm}
[Sumner and Moore \cite{Sum1,Sum2}]   \label{Th5}
A planar  graph  is  domination
perfect if and only if it does not contain any graph from ${\cal A}$ as  an  induced
subgraph.
\end{thm}

Bollob\'as and Cockayne \cite{Bol} generalized the result of  Allan
and Laskar (Theorem~\ref{A-L}) as follows:

\begin{thm}
[Bollob\'as and Cockayne \cite{Bol}] 
If $G$ has no induced  subgraph
isomorphic to $K_{1,k}$ $(k\ge 3)$, then
\begin{equation}
\label{In1}
i(G) \le  \gamma (G)(k-2) - (k-3).
\end{equation}
\end{thm}

It was proved in \cite{Zve0} 
that inequality (\ref{In1})
in fact holds
for a wider class  of  graphs:

\begin{thm}
[Zverovich and Zverovich \cite{Zve0}]
Suppose $G$ does not  contain
two induced subgraphs $K_{1,k}$ $(k\ge 3)$ having  different  centers and exactly one edge in common. 
Then inequality (\ref{In1}) holds.
\end{thm}

For $k=3,$ we obtain the following result, which provides the
first four minimal domination imperfect graphs of order six ($G_{1}-G_{4}$ in Figure \ref{17graphs}).

\begin{cor}
[Zverovich and Zverovich \cite{Zve0}]
\label{GiTi}
If a graph $G$ does not contain the graphs $G_{1}-G_{4}$  in
Figure \ref{17graphs}
and the graphs $T_{1},T_{2}$ in Figure \ref{T1T2} as induced subgraphs, then $G$  is
domination perfect.
\end{cor}

\vspace*{0.1cm}
\begin{figure}[h!] \centering\includegraphics[width=0.4\linewidth]{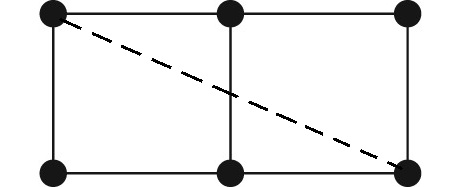} 
\caption{Graphs $T_1$ and $T_2$.}
\label{T1T2} 
\end{figure}

Note that Theorems \ref{A-L} and 
\ref{Chordal} follow directly from Corollary \ref{GiTi}.
It can be shown that 
a triangle-free graph is
domination perfect if and only if 
it contains no induced subgraphs isomorphic to
$G_{1}-G_{4}$ \cite{Zve0}.
Fulman \cite{Ful} gave a sufficient condition for a graph to be domination perfect in terms of eight forbidden graphs (see Theorem~\ref{Ful}).

Topp  and  Volkmann  \cite{Top} identified  the  13
minimal domination imperfect graphs among the graphs in the
family ${\cal A}$, which is used in  Theorems \ref{Th4} and \ref{Th5}.

\begin{thm}
[Topp and Volkmann \cite{Top}] 
\label{T-V}
If a graph $G$  does  not  contain
any of the graphs $G_{1}-G_{13}$ in Figure \ref{17graphs} and 
the graph $S$ in  
Figure \ref{Sgraph}  as  an
induced subgraph, then $G$ is domination perfect.
\end{thm}

Sumner \cite{Sum2} suggested that it is impossible to provide a finite forbidden-subgraph characterization for the entire class of domination perfect graphs. Nevertheless, in 1995, the following characterization in terms of 17 forbidden subgraphs was proposed:

\begin{thm}[Zverovich and Zverovich \cite{Zve1}]
\label{ZZ}
A graph $G$ is domination perfect if  and  only  if $G$
does not contain any of the graphs $G_{1}-G_{17}$  in  Figure  \ref{17graphs}  as  an
induced subgraph.
\end{thm}

\vspace{-0.3cm}

\begin{figure}[b!]
\begin{large}
\setlength{\unitlength}{1cm}
\thicklines
\hspace*{-1cm}
\begin{picture}(17,17)

\put(1,15.5){\circle*{0.2}}
\put(1,17){\circle*{0.2}}
\put(2.5,15.5){\circle*{0.2}}
\put(2.5,17){\circle*{0.2}}
\put(4,15.5){\circle*{0.2}}
\put(4,17){\circle*{0.2}}

\put(5,15.5){\circle*{0.2}}
\put(5,17){\circle*{0.2}}
\put(6.5,15.5){\circle*{0.2}}
\put(6.5,17){\circle*{0.2}}
\put(8,15.5){\circle*{0.2}}
\put(8,17){\circle*{0.2}}

\put(9,15.5){\circle*{0.2}}
\put(9,17){\circle*{0.2}}
\put(10.5,15.5){\circle*{0.2}}
\put(10.5,17){\circle*{0.2}}
\put(12,15.5){\circle*{0.2}}
\put(12,17){\circle*{0.2}}

\put(13,15.5){\circle*{0.2}}
\put(13,17){\circle*{0.2}}
\put(14.5,15.5){\circle*{0.2}}
\put(14.5,17){\circle*{0.2}}
\put(16,15.5){\circle*{0.2}}
\put(16,17){\circle*{0.2}}

\drawline(1,15.5)(4,15.5)(2.5,15.5)(2.5,17)(1,17)(4,17)

\drawline(5,17)(5,15.5)(8,15.5)(6.5,15.5)(6.5,17)(5,17)(8,17)

\drawline(12,15.5)(9,17)(9,15.5)(12,15.5)(10.5,15.5)(10.5,17)(9,17)(12,17)

\drawline(13,15.5)(16,17)(16,15.5)(13,17)(13,15.5)(16,15.5)(14.5,15.5)(14.5,17)(13,17)(16,17)

\put(1,12.5){\circle*{0.2}}
\put(1,14){\circle*{0.2}}
\put(2.5,12.5){\circle*{0.2}}
\put(2.5,14){\circle*{0.2}}
\put(4,12.5){\circle*{0.2}}
\put(4,14){\circle*{0.2}}
\put(3.25,13.25){\circle*{0.2}}

\put(5,12.5){\circle*{0.2}}
\put(5,14){\circle*{0.2}}
\put(6.5,12.5){\circle*{0.2}}
\put(6.5,14){\circle*{0.2}}
\put(8,12.5){\circle*{0.2}}
\put(8,14){\circle*{0.2}}
\put(7.25,13.25){\circle*{0.2}}

\put(9,12.5){\circle*{0.2}}
\put(9,14){\circle*{0.2}}
\put(10.5,12.5){\circle*{0.2}}
\put(10.5,14){\circle*{0.2}}
\put(12,12.5){\circle*{0.2}}
\put(12,14){\circle*{0.2}}
\put(9.75,13.25){\circle*{0.2}}
\put(11.25,13.25){\circle*{0.2}}

\put(13,12.5){\circle*{0.2}}
\put(13,14){\circle*{0.2}}
\put(14.5,12.5){\circle*{0.2}}
\put(14.5,14){\circle*{0.2}}
\put(16,12.5){\circle*{0.2}}
\put(16,14){\circle*{0.2}}
\put(13.75,13.25){\circle*{0.2}}
\put(15.25,13.25){\circle*{0.2}}

\drawline(1,12.5)(1,14)(4,14)(4,12.5)(1,12.5)
\drawline(2.5,12.5)(2.5,14)(3.25,13.25)(2.5,12.5)

\drawline(5,12.5)(5,14)(8,14)(8,12.5)(5,12.5)
\drawline(6.5,12.5)(6.5,14)(8,12.5)(5,14)(7.25,13.25)(6.5,12.5)

\drawline(9,12.5)(9,14)(12,14)(12,12.5)(9,12.5)
\drawline(12,12.5)(10.5,14)(9.75,13.25)(12,14)
\drawline(9.75,13.25)(10.5,12.5)(10.5,14)(10.5,12.5)(11.25,13.25)

\drawline(13,12.5)(13,14)(16,14)(16,12.5)(13,12.5)
\drawline(16,12.5)(14.5,14)(13,12.5)
\drawline(13.75,13.25)(14.5,12.5)(14.5,14)(14.5,12.5)(15.25,13.25)

\put(1,9.5){\circle*{0.2}}
\put(1,11){\circle*{0.2}}
\put(2.5,9.5){\circle*{0.2}}
\put(2.5,11){\circle*{0.2}}
\put(4,9.5){\circle*{0.2}}
\put(4,11){\circle*{0.2}}
\put(1.75,10.25){\circle*{0.2}}
\put(3.25,10.25){\circle*{0.2}}

\put(5,9.5){\circle*{0.2}}
\put(5,11){\circle*{0.2}}
\put(6.5,9.5){\circle*{0.2}}
\put(6.5,11){\circle*{0.2}}
\put(8,9.5){\circle*{0.2}}
\put(8,11){\circle*{0.2}}
\put(5.75,10.25){\circle*{0.2}}
\put(7.25,10.25){\circle*{0.2}}

\put(9,9.5){\circle*{0.2}}
\put(9,11){\circle*{0.2}}
\put(10.5,9.5){\circle*{0.2}}
\put(10.5,11){\circle*{0.2}}
\put(12,9.5){\circle*{0.2}}
\put(12,11){\circle*{0.2}}
\put(9.75,10.25){\circle*{0.2}}
\put(11.25,10.25){\circle*{0.2}}

\put(13,9.5){\circle*{0.2}}
\put(13,11){\circle*{0.2}}
\put(14.5,9.5){\circle*{0.2}}
\put(14.5,11){\circle*{0.2}}
\put(16,9.5){\circle*{0.2}}
\put(16,11){\circle*{0.2}}
\put(13.75,10.25){\circle*{0.2}}
\put(15.25,10.25){\circle*{0.2}}

\drawline(1,9.5)(1,11)(4,11)(4,9.5)(1,9.5)
\drawline(4,9.5)(2.5,11)(1,9.5)
\drawline(4,11)(1.75,10.25)(2.5,9.5)(2.5,11)(2.5,9.5)(3.25,10.25)

\drawline(5,9.5)(5,11)(8,11)(8,9.5)(5,9.5)
\drawline(8,9.5)(6.5,11)(5.75,10.25)(8,11)
\drawline(5.75,10.25)(6.5,9.5)(6.5,11)(6.5,9.5)(7.25,10.25)(5.75,10.25)

\drawline(9,9.5)(9,11)(12,11)(12,9.5)(9,9.5)
\drawline(12,9.5)(10.5,11)(9,9.5)
\drawline(9.75,10.25)(10.5,9.5)(10.5,11)(10.5,9.5)(11.25,10.25)(9.75,10.25)

\drawline(13,9.5)(13,11)(16,11)(16,9.5)(13,9.5)
\drawline(16,9.5)(14.5,11)(13,9.5)
\drawline(13.75,10.25)(14.5,9.5)(14.5,11)(14.5,9.5)(15.25,10.25)(13.75,10.25)(16,11)

\put(1,6){\circle*{0.2}}
\put(1,7.5){\circle*{0.2}}
\put(2.5,6){\circle*{0.2}}
\put(2.5,7.5){\circle*{0.2}}
\put(4,6){\circle*{0.2}}
\put(4,7.5){\circle*{0.2}}
\put(1.75,6.75){\circle*{0.2}}
\put(3.25,6.75){\circle*{0.2}}

\put(5,6){\circle*{0.2}}
\put(5,7.5){\circle*{0.2}}
\put(6.5,6){\circle*{0.2}}
\put(6.5,7.5){\circle*{0.2}}
\put(8,6){\circle*{0.2}}
\put(8,7.5){\circle*{0.2}}
\put(5.75,6.75){\circle*{0.2}}
\put(7.25,6.75){\circle*{0.2}}
\put(5.75,8.5){\circle*{0.2}}

\put(9,6){\circle*{0.2}}
\put(9,7.5){\circle*{0.2}}
\put(10.5,6){\circle*{0.2}}
\put(10.5,7.5){\circle*{0.2}}
\put(12,6){\circle*{0.2}}
\put(12,7.5){\circle*{0.2}}
\put(9.75,6.75){\circle*{0.2}}
\put(11.25,6.75){\circle*{0.2}}
\put(9.75,8.5){\circle*{0.2}}

\put(13,6){\circle*{0.2}}
\put(13,7.5){\circle*{0.2}}
\put(14.5,6){\circle*{0.2}}
\put(14.5,7.5){\circle*{0.2}}
\put(16,6){\circle*{0.2}}
\put(16,7.5){\circle*{0.2}}
\put(13.75,6.75){\circle*{0.2}}
\put(15.25,6.75){\circle*{0.2}}
\put(14.5,8.5){\circle*{0.2}}

\drawline(1,6)(1,7.5)(4,7.5)(4,6)(1,6)
\drawline(4,6)(2.5,7.5)(1,6)
\drawline(1.75,6.75)(2.5,6)(2.5,7.5)(2.5,6)(3.25,6.75)
\drawline(1,7.5)(3.25,6.75)(1.75,6.75)(4,7.5)

\drawline(5,6)(5,7.5)(8,7.5)(8,6)(5,6)(5.75,8.5)(6.5,6)
\drawline(8,6)(6.5,7.5)(5,6)
\drawline(6.5,7.5)(5.75,8.5)(8,7.5)(5.75,6.75)(6.5,6)(6.5,7.5)(6.5,6)(7.25,6.75)(5,7.5)

\drawline(9,6)(9,7.5)(12,7.5)(12,6)(9,6)(9.75,8.5)(10.5,6)
\drawline(12,6)(10.5,7.5)(9,6)
\drawline(12,7.5)(9.75,6.75)(10.5,6)(10.5,7.5)(10.5,6)(11.25,6.75)(9,7.5)(9.75,8.5)(10.5,7.5)

\drawline(13,6)(13,7.5)(16,7.5)(16,6)(13,6)(14.5,8.5)
\drawline(16,6)(14.5,7.5)(13,6)
\drawline(16,6)(14.5,8.5)(16,7.5)(13.75,6.75)(13.75,6.75)(14.5,6)(14.5,7.5)(14.5,6)(15.25,6.75)(13,7.5)(14.5,8.5)(14.5,7.5)

\put(5.5,2){\circle*{0.2}}
\put(5.5,4){\circle*{0.2}}
\put(7,1){\circle*{0.2}}
\put(7,3){\circle*{0.2}}
\put(8.5,2){\circle*{0.2}}
\put(8.5,4){\circle*{0.2}}
\put(8.5,5){\circle*{0.2}}
\put(10,1){\circle*{0.2}}
\put(10,3){\circle*{0.2}}
\put(11.5,2){\circle*{0.2}}
\put(11.5,4){\circle*{0.2}}

\drawline(8.5,4)(11.5,2)(5.5,2)(5.5,4)(11.5,4)(8.5,5)(8.5,2)(10,3)(10,1)(7,3)(8.5,5)(5.5,4)(7,1)(10,1)(11.5,4)(11.5,2)(10,1)(8.5,4)(5.5,2)(7,1)(8.5,4)(5.5,4)(10,3)(8.5,5)(11.5,2)(11.5,4)(7,3)

\bezier{260}(14.5,5.95)(13.95,7.1)(14.5,8.5) 
\bezier{260}(7,1)(9.5,2)(11.5,4) 

\put(2.3,14.8){$G_1$}
\put(2.3,11.8){$G_5$} 
\put(2.3,8.8){$G_9$}
\put(2.3,5.3){$G_{13}$}

\put(6.3,14.8){$G_2$}
\put(6.3,11.8){$G_6$} 
\put(6.3,8.8){$G_{10}$}
\put(6.3,5.3){$G_{14}$}

\put(10.3,14.8){$G_3$}
\put(10.3,11.8){$G_7$} 
\put(10.3,8.8){$G_{11}$}
\put(10.3,5.3){$G_{15}$}

\put(14.3,14.8){$G_4$}
\put(14.3,11.8){$G_8$} 
\put(14.3,8.8){$G_{12}$}
\put(14.3,5.3){$G_{16}$}

\put(8.3,0.3){$G_{17}$}

\end{picture}
\end{large}
\vspace*{-0.4cm}
\caption{Graphs $G_1-G_{17}$.}
\label{17graphs}
\end{figure}
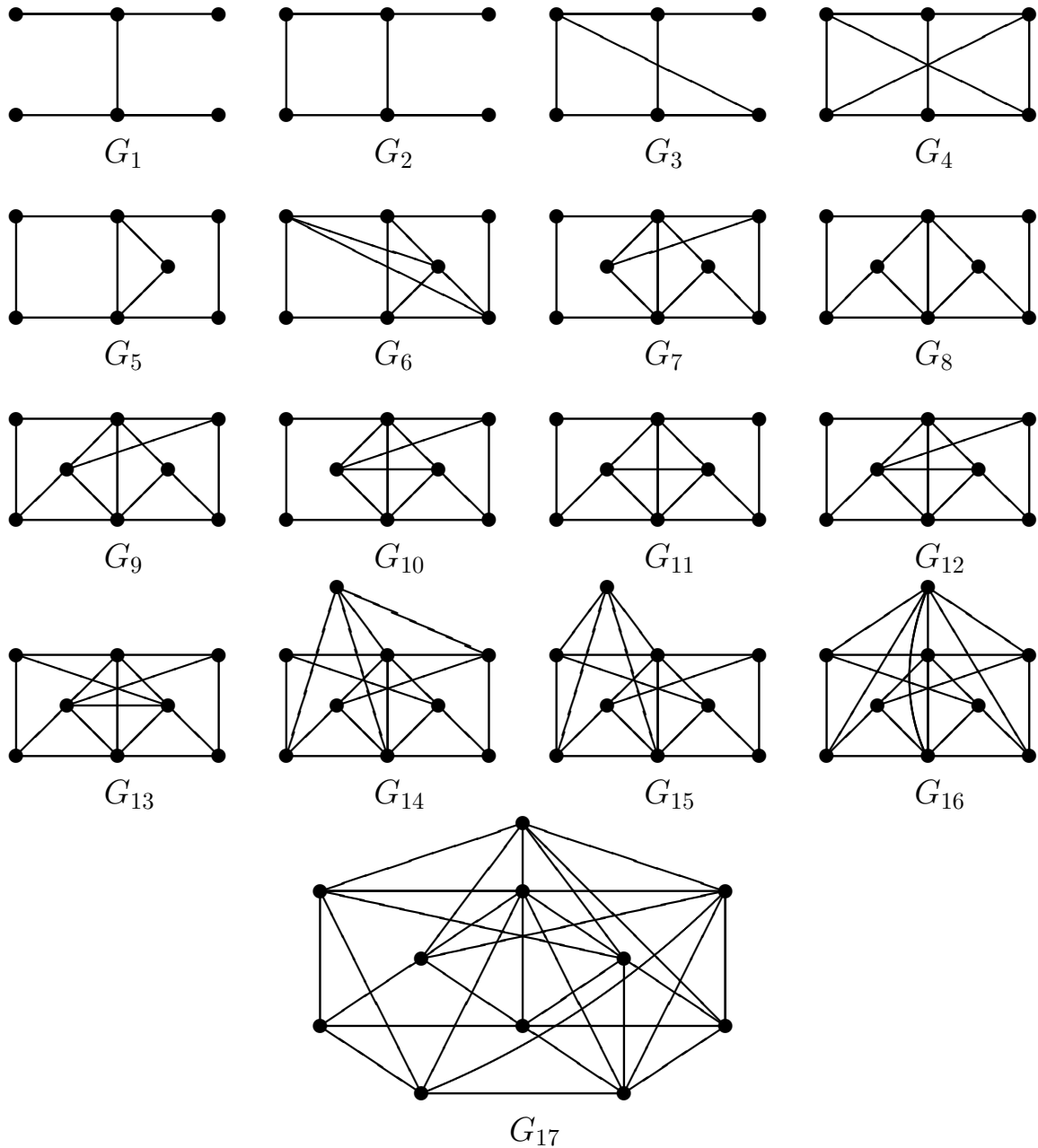

\begin{cor}
[Zverovich and Zverovich \cite{Zve1}]
\label{Cor2}
Suppose that a graph $G$ does not contain the complements  of  the
graphs $G_{1}-G_{17}$ 
in Figure \ref{17graphs} 
as induced  subgraphs,  and that  ${\rm diam}(G)>2.$
Then $G$ has an edge that does not belong to any triangle.
 \end{cor}

We say that a graph $G$  belongs  to  the  class ${\cal L}$  if  the
following conditions hold:
\be
\up\item[(i)] $G$ is planar;
\up\item[(ii)] $G$ is bipartite;
\up\item[(iii)] $G$ has maximum degree 3;
\up\item[(iv)] $G$ has girth $g(G)\ge k$, where $k$ is fixed.
\ee

\begin{cor}
[Zverovich and Zverovich \cite{Zve1}]
\label{Cor3}
The Dominating Set and  Inde\-pen\-dent
Dominating Set problems 
are both $NP$-complete on the
class ${\cal L}$.
\end{cor}

\subsection*{On Camby--Plein's Paper}

In 2017, Camby and Plein 
(\cite{Cam}, p.~714)
claimed that the graphs $H_5$ and $H_6$ in Figure \ref{CPDiagram}
are counterexamples to Theorem 
\ref{ZZ}: 
``Indeed, none contains any 
$G_1, G_2, ..., G_{17}$
as an induced subgraph, as the system GraphsInGraphs [-] confirmed. And none is domination perfect, since 
$\ga(H_5) = \ga(H_6) =
2 \not= 3 = i(H_5) = i(H_6).$"

\vspace*{0.2cm}
\begin{figure}[b!] \centering\includegraphics[width=0.5\linewidth]{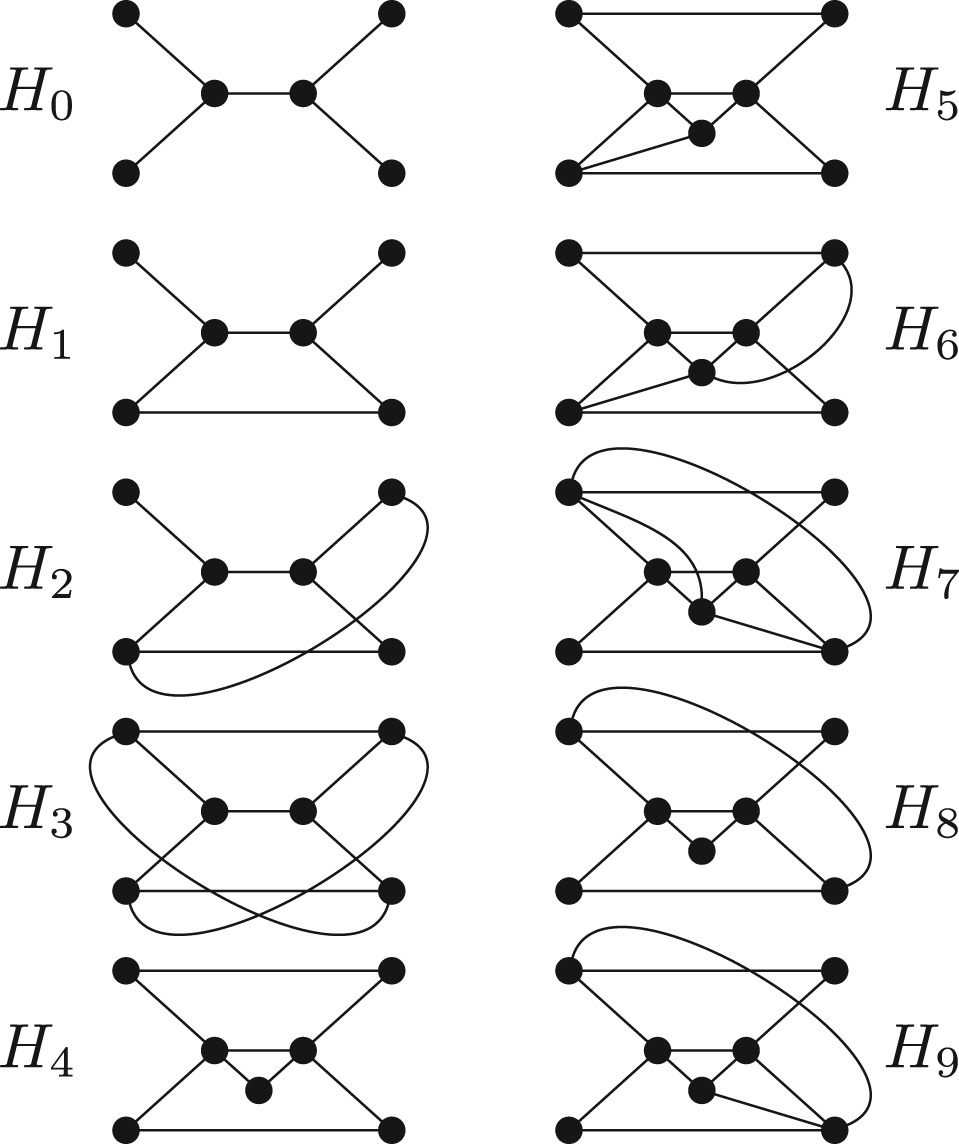} 
\caption{Graphs $H_0-H_9$ \cite{Cam}.}
\label{CPDiagram} 
\end{figure}

In the graph $H_5$, let $x$ denote the top-right vertex and
let $y$ denote the bottom-left vertex. Then, $\{x,y\}$ is an
independent dominating set.
Because $H_5$ does not have a
dominating vertex, we obtain
$i(H_5) = 2.$
In a similar way, 
$i(H_6) = 2.$
Thus, the above `proof' that the graphs $H_5$ and $H_6$ in Figure \ref{CPDiagram}
are counterexamples to Theorem 
\ref{ZZ} is incorrect.
In fact, the following statement can be proved directly; it also follows from Theorem~\ref{ZZ}.

\begin{prop}
\label{Prop1}
The graphs $H_5$ and $H_6$ in Figure \ref{CPDiagram}
are domination perfect.    
\end{prop}

Furthermore, the following result was proved in \cite{Cam} (p.~715):

\begin{thm}
[Camby and Plein \cite{Cam}]
\label{Th8}
Let $G$ be a $\{H_0,...,H_7\}$-free graph, where the graphs $H_k$ are shown in Figure~\ref{CPDiagram}.
Then $$i(G) = \ga(G).$$    
\end{thm}

However, Theorem~\ref{Th8} is a known result published in 1993 (and cited in \cite{Cam} as Theorem~5):

\begin{thm}
[Fulman \cite{Ful}]
\label{Ful}
Let $G$ be a $\{H_0,...,H_4, H_7, U_1, U_2\}$-free graph, where the graphs $U_1$ and $U_2$ are shown in Figure~\ref{U1U2}.
Then $G$ is a domination perfect graph; in particular,
$$i(G) = \ga(G).$$  
\end{thm}

\begin{figure}[h!] \centering\includegraphics[width=0.6\linewidth]{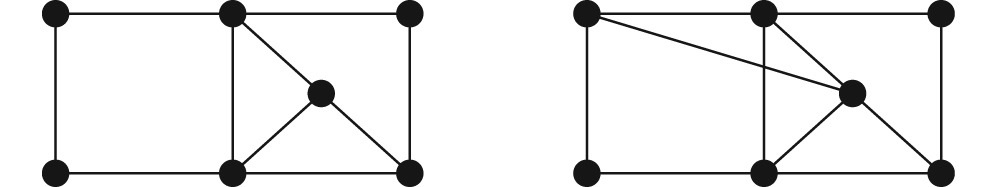} 
\caption{Graphs $U_1$ (left) and $U_2$ (right).}
\label{U1U2} 
\end{figure}

Indeed, it is easy to see that the graph $U_1$ is isomorphic to the graph $H_5$, and the graph $U_2$ is isomorphic to the graph $H_6$.
Therefore, Theorem~\ref{Th8} is equivalent to Theorem~\ref{Ful}.
The proof of Theorem \ref{Th8}, which is based on Lemma 1 of \cite{Cam}, is also known.

It may be pointed out that Theorems \ref{Th8} and \ref{Ful} follow from Theorem \ref{ZZ} because any graph $G_m$, $1\le m\le 17$, contains 
at least one of the graphs $H_0,..., H_7$ as an induced subgraph.
Indeed, $G_1,...,G_5$ are isomorphic to 
$H_0,...,H_4$, respectively; $G_6$ is isomorphic to $H_7$; $G_7,...,G_{12}$ contain an induced subgraph isomorphic to $H_5$; and
$G_{13},...,G_{17}$ contain an induced subgraph isomorphic to $H_6$.

Using Theorem \ref{Th8},
a `new characterization' of domination perfect graphs was given in \cite{Cam} (p.~716):

\begin{thm}
[Camby and Plein \cite{Cam}]
\label{NewChar}
The following assertions are equivalent for every graph $G$:
\begin{itemize}
\up\item $G$ is a $\{H_0,...,H_7\}$-free.
\up\item $G$ is a $\{H_0,...,H_9\}$-free.
\up\item For every induced subgraph $H$ 
of $G$, it holds that $i(H) = \ga(H).$    
\end{itemize}
\end{thm}

This characterization is not correct. 
Indeed, by Proposition \ref{Prop1},
the graph $H_5$ is domination perfect, and thus the third condition of Theorem \ref{NewChar} is satisfied for $H_5$.
However, the graph $H_5$ contains itself as an induced subgraph
and hence is not $H_5$-free.
The graph $H_6$ provides another counterexample to this characterization, and additional counterexamples can be constructed.
Moreover, it is unclear why the class of $\{H_0,...,H_9\}$-free graphs is included in the theorem, 
since this class coincides with 
$\{H_0,...,H_7\}$-free graphs:
$H_8$ contains $H_1$ and 
$H_9$ contains $H_2$ as an induced subgraph.

Furthermore, using Theorem \ref{Th8}, Camby and Plein \cite{Cam} attempted to `improve' some known results.
For instance, they reformulated 
Corollary \ref{Cor2}
for $\{\overline H_0,...,\overline H_7\}$-free graphs, i.e.,  
for a subclass of $\{\overline G_1,...,\overline G_{17}\}$-free graphs, 
thereby obtaining a weaker result.
Finally, Camby and Plein \cite{Cam} attempted to `rectify' Corollary 
\ref{Cor3} and the next theorem, whose original proofs are valid.

Let ${\cal M}$ be the class of graphs $G$ such that, for every induced
subgraph $H$ of $G$, $H$ has a unique minimum irredundant set (see \cite{Bol}) if and only if it has a unique minimum dominating set.

\begin{thm}
[Fischermann, Volkmann, Zverovich 
\cite{Fis}] 
A graph $G$ belongs to 
the class ${\cal M}$ if and only if $G$ is 
$\{B_1, B_2, B_3\}$-free.
\end{thm}

\begin{figure}[h!] \centering\includegraphics[width=0.5\linewidth]{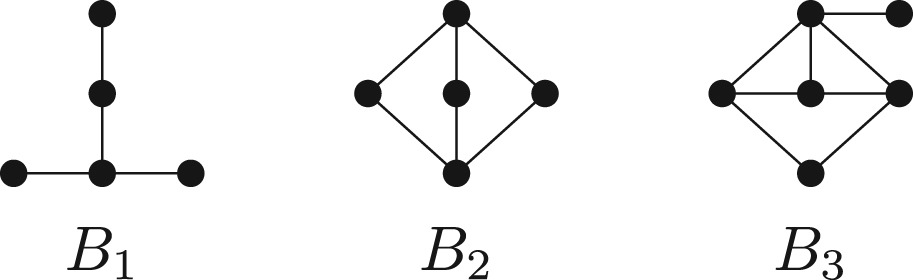} 
\caption{Graphs $B_1-B_3$.}
\label{B1B3} 
\end{figure}

\section{Proof of Theorem \ref{ZZ}}

In this section, we present a relatively concise proof of Theorem \ref{ZZ},
which is significantly shorter than the original proof in \cite{Zve1}.

The necessity follows from the  fact  that $\gamma (G_{j})=2$
and $i (G_{j})=3$ for $1\le j\le 17$.  
To prove  the  sufficiency, let $F$  be  a
minimum counterexample, i.e., $\gamma (F) < i (F)$, the  graph $F$  does  not
contain any of the induced subgraphs $G_{1}-G_{17}$ in Figure \ref{17graphs},  and $F$ is of minimum order. 
Let $D$ be a minimum dominating set of $F$ such
that the number of edges in $\langle D\rangle$ is minimum. Since $\vert D\vert =\gamma (F) < i (F)$
and $D$ dominates $F,$ it follows that the set $D$  is  dependent  and
hence contains some edge $uv$. Let
\begin{eqnarray*}
A &=& \{x\in V(F)\backslash D: N(x)\cap D=\{u\}\},\\
B &=& \{x\in V(F)\backslash D: N(x)\cap D=\{v\}\},\\
C &=& \{x\in V(F)\backslash D: N(x)\cap D=\{u,v\}\}.  
\end{eqnarray*}

%
%
\par
We need the following lemmas.

\up\begin{lem}
The graph $F$ satisfies $i (F)>2$ and
$V(F)=A\cup B\cup C\cup \{u,v\}$.
\end{lem}

\up\pf We first prove that 
$i(H)>2$, where $H=\langle A\cup B\cup C\cup \{u,v\}\rangle$.
Indeed, if $i (H)=1$ and $x$ is a dominating vertex of $H$, then $D$  is
not a minimum dominating set for $F$, since  the  set $D\backslash \{u,v\}\cup \{x\}$
dominates $F$. If $i (H)=2$ and $\{x,y\}$ is  an  independent  dominating
set of $H$, then $D\backslash \{u,v\}\cup \{x,y\}$ is a minimum dominating set  of $F$
whose induced subgraph contains fewer edges than $\langle D\rangle$, 
contradicting the
hypothesis.
Thus $\gamma (H)=2$, $i(H)>2$ and $H$ does  not  contain  any  of  the
induced subgraphs $G_{1}-G_{17}$. Since $F$ is a  minimum  counterexample,
we obtain $F=H$.  \qed

The vertex set of any of the graphs $G_{1}-G_{17}$ in Figure \ref{17graphs} can be represented in the form 
$V(F)=A\cup B\cup C\cup \{u,v\}$, since $\gamma (G_{j})=2$ for  all $1\le j\le 17$.  For  uniformity,  the
graph $G_{j}$ in such a representation will be cited as follows:
$$
\langle u,\hbox{vertices of } A;\; v,\hbox{vertices of } B;\;\hbox{vertices of } C\rangle.
$$

\begin{lem}
Let $a,b\in A$ and $c,d\in B$ be arbitrary vertices.
\be
\up\item[(i)] If $a\nssim c$, then $\{a,c\}$ dominates the graph $\langle A\cup B\rangle$;
\par
\up\up\item[(ii)] If $a\nssim b$ and $c\nssim d$, then $\langle a,b,c,d\rangle \simeq  2K_{2}$;
\up\up\item[(iii)] The graph $\langle N(a)\cap B\rangle$  $(\langle N(c)\cap A\rangle)$ is complete.
\ee
\end{lem}

\up\pf (i) Assume to the  contrary  that $\{a,c\}$  does  not
dominate $\langle A\cup B\rangle$. Then, there  exists  a  vertex $s\in A\cup B$  such  that
$s\nssim \{a,c\}$. Without loss of generality, let $s\in A$. By  Lemma  1,  the
set $\{u,c\}$ does not dominate $F$, hence there is a vertex $t\in B$  such
that $t\nssim c$. Consider the graph $\langle u,a,s;v,c,t\rangle$. Depending  on  the  existence  of
edges $at$ and $st$, this graph is  isomorphic  to $G_{1}$, $G_{2}$  or $G_{3}$, a
contradiction.
\par
(ii) Assume to the contrary that $\langle a,b,c,d\rangle\not\simeq 2K_{2}$.  Then,  if
$\langle a,b,c,d\rangle$ contains 0, 1 or 2 edges, we have  a  contradiction  to (i).
Therefore, $\langle a,b,c,d\rangle \simeq  P_{4}$  or $C_{4}$.  If $\langle a,b,c,d\rangle\simeq C_{4}$  then
$\langle u,a,b;v,c,d\rangle\simeq G_{4}$, a contradiction. Thus, $\langle a,b,c,d\rangle\simeq P_{4}$,  and  let
$b\nssim c$. By Lemma 1, the set $\{b,c\}$ does not dominate $F$, i.e.,  there
is a vertex $f\in V(F)$ such that $f\nssim \{b,c\}$.
We further know that $\{b,c\}$ dominates $\langle A\cup B\rangle$, so $f\in C$.
Let $L$ denote the  resulting  graph. 
If $f\nssim \{a,d\}$, then $L\backslash v\simeq G_{2}$.
If $f\nssim a$ and $f\ssim d$ ($f\ssim a$  and
$f\nssim d$), then $L\backslash v\simeq G_{3}$ ($L\backslash u\simeq
G_{3}$). 
Finally, if $f\ssim \{a,d\}$,  then $L\simeq G_{6}$.  In
all cases a contradiction is obtained.
\par
(iii) Suppose that $a\ssim \{c,d\}$  and $c\nssim d$. By Lemma  1,  since
$\{a,v\}$ does not dominate $F$, there is a vertex $b\in A$ such that
$b\nssim a$.
Then $\langle a,b,c,d\rangle  \not\simeq  2K_{2}$, contrary to Lemma 2(ii).  Consequently,
$c\ssim d$.       \qed

\vspace*{-0.3cm}
\begin{lem}
Let $a,b\in A$, \,$c,d\in B$ and $f\in C$.
If $a\nssim b$ and $c\nssim d$,  then $f$  is
adjacent to at least two vertices of the set $\{a,b,c,d\}$.
\end{lem}

\up\pf Suppose to the contrary that $f$ is  adjacent  to  at
most  one  vertex in  the  set $\{a,b,c,d\}$.  Without  loss  of
generality,  by  Lemma 2(ii), $a\ssim c$, $b\ssim d$  and $a\nssim
d$, $b\nssim c$.  If
$f\nssim \{a,b,c,d\}$, then $\langle u,a,b;v,c,d;f\rangle  \simeq  G_{5}$, a contradiction.
Hence, $f$
is adjacent to exactly one vertex from $\{a,b,c,d\}$, say  $f\ssim b$. By Lemma  1,  the  set $\{b,c\}$  does  not
dominate $F$, i.e., there is a vertex $g\in V(F)$ such that $g\nssim
\{b,c\}$.
By Lemma 2(i), $g\in C$.  If $g\nssim \{a,d\}$,  then
$\langle u,a,b;v,c,d;g\rangle \simeq G_{5}$, a contradiction.
Let us denote the resulting graph by $M$.  
If $f\nssim g$, then $M$ is isomorphic to
one of $G_7-G_9$, 
and if $f\ssim g$, then 
$M$ is isomorphic to one of $G_{10}-G_{12}$,
a contradiction.  \qed

\vspace*{-0.3cm}
\begin{lem}
Let $a,b\in A$, \,$c,d\in B$ and $f,g\in C$. If $\langle a,b,c,d,f\rangle $ induces a $P_{5}$  with
path $(a-c-f-b-d)$ and $g\nssim \{b,c\}$,  then $g\ssim \{a,d\}$
and $g\nssim f$.
\end{lem}

\up\pf Apply Lemma 3 to the set $\{a,b,c,d\}$ and  the  vertex
$g$.  Since $g\nssim \{b,c\}$,  we  have $g\ssim \{a,d\}$.   If   now $g\ssim f$,   then
$\langle u,a,b;v,c,d;f,g\rangle $  is  isomorphic  to $G_{13}$, a   contradiction.
Consequently, $g\nssim f$.       \qed

\vspace*{-0.4cm}
\begin{lem}
Let $a,b\in A$, \,$c,d\in B$, \,$f,g\in C$ and $h\in V(F)$. 
If $\langle a,b,c,d,f,g\rangle$ 
induces a $C_{6}$
with vertex order $(a-c-f-b-d-g-a)$
and $h\nssim \{f,g\}$,  then  only  four  cases  are
possible:

\smallskip
(1)\; $h\in A$,~ $h\ssim \{a,b,c\}$,~ $h\nssim d$;
\par
(2)\; $h\in A$,~ $h\ssim \{a,b,d\}$,~ $h\nssim c$;
\par
(3)\; $h\in B$,~ $h\ssim \{c,d,a\}$,~ $h\nssim b$;
\par
(4)\; $h\in B$,~ $h\ssim \{c,d,b\}$,~ $h\nssim a$.
\end{lem}

\up\pf Let us assume that $h\in C$. By Lemma 3, $h$  is  adjacent
to at least two vertices  of $\{a,b,c,d\}$.  Taking  symmetry  into
account, we have the following 
possibilities:

(a) if $h\ssim \{a,b\}$ and $h\nssim \{c,d\}$, then 
$\langle h,a,b;v,d,f
\rangle \simeq  G_{3}$;

(b) if $h\ssim \{b,c\}$ and $h\nssim \{a,d\}$, then $\langle u,a,b;v,c,d;f,g,h
\rangle\simeq  G_{14}$;

(c) if $h\ssim \{b,d\}$ and $h\nssim \{a,c\}$, then $\langle u,a,b;v,c,d;f,g,h
\rangle\simeq  G_{15}$;

(d) if $h\ssim \{a,b,d\}$ and $h\nssim c$, then 
$\langle h,a,b;v,c,g
\rangle \simeq  G_{3}$;

(e) if $h\ssim \{a,b,c,d\}$, then $\langle u,a,b;v,c,d;f,g,h
\rangle\simeq  G_{16}$.

Thus, $h\in A\cup B$.
Assume $h\in A$; the case $h\in B$ is similar.  
By
Lemma 2(i), the vertex $h$ must  be  adjacent  to  at  least  one
vertex of $\{c,d\}$.

\textit{Case 1.} Suppose $h\ssim c$. By Lemma 2(iii), we have $h\nssim d$  and
$h\ssim a$.
Assume that $h\nssim b$. Then, by Lemma 3, $g$  must  be  adjacent  to  at
least two vertices  of $\{h,b,c,d\}$. This is a contradiction, since
$g\nssim \{h,b,c\}$. Consequently, $h\ssim b$.

\textit{Case 2.} Suppose $h\ssim d$. By Lemma 2(iii), we have $h\nssim c$  and
$h\ssim b$.
Assume that $h\nssim a$. Then, by Lemma 3, $f$  must  be  adjacent  to  at
least two vertices  of $\{a,h,c,d\}$. This is a contradiction, since
$f\nssim \{a,h,d\}$. Consequently, $h\ssim a$.
\qed

We proceed with the proof of Theorem \ref{ZZ}.
By  Lemma  1,
there exist vertices $a,b\in A$ and $c,d\in B$ such that $a\nssim b$
and $c\nssim d$.  By
Lemma 2(ii), $\langle a,b,c,d\rangle \simeq 2K_{2}$. Without loss of generality, let $a\ssim c$
and $b\ssim d$. In accordance with Lemma 1,  the  set $\{a,d\}$  does  not
dominate $F$. Therefore,  there  is  a  vertex $f\in V(F)$  such  that
$f\nssim \{a,d\}$. By Lemma 2(i), $f\in C$.  Then, $f\ssim \{b,c\}$  by  Lemma  3.  Now,
consider the vertices $b$ and $c$. Again, by  Lemma  1,  there  is  a
vertex $g\in V(F)$ such that $g\nssim \{b,c\}$. By Lemma 2(i), we have $g\in C$ and,
by Lemma 4, $g\ssim \{a,d\}$ and $g\nssim f$. 
Thus,
$\langle a,b,c,d,f,g\rangle$ 
induces a $C_{6}$
with vertex order $(a-c-f-b-d-g-a)$.

Consider the set $\{f,g\}$. By Lemma 1, there exists a  vertex
$h\in V(F)$ such that $h\nssim \{f,g\}$. In accordance with Lemma 5, there  are
four cases. These cases are symmetrical, so we can  assume  that
$h\in A$ and $h\ssim \{a,b,c\}$, $h\nssim d$.  By  Lemma  1,  since $\{h,v\}$  does  not
dominate $F$, there exists a vertex $k\in A$ such that $k\nssim h$.
By  Lemma 2(ii),
$\langle h,k,c,d\rangle \simeq 2K_{2}$, hence $k\ssim d$ and $k\nssim c$. Then, $k\ssim b$ by Lemma 2(iii). For
the set $\{h,k,c,d\}$ and $f,g$, we can apply Lemma 3, which  produces
$k\ssim \{f,g\}$. 
Thus, the only 
undetermined edge is $ka$.
Denote the resulting graph by $Q$, obtained from the graph in Figure~\ref{F7}(a) by deleting the vertex 
$l$ ($u$ and $v$ are omitted).

In accordance  with  Lemma  1,  the  set $\{k,c\}$  does  not
dominate $F$. Therefore,  there  is  a  vertex $l\in V(F)$  such  that
$l\nssim \{k,c\}$. By Lemma 2(i), $l\in C$.  Applying  Lemma  4  to  the  set
$\{h,k,c,d,f\}$ and the vertex $l$, we obtain $l\ssim \{h,d\}$  and
$l\nssim f$.  
We now apply Lemma 5. Assume  that $l\nssim g$.  Then, $\{h,k,c,d,f,l\}$  and  the
vertex $g$ satisfy the hypotheses of Lemma 5. By the conclusion of
Lemma 5, we have $g\not\in C$, a contradiction. Thus, $l\ssim g$.

Suppose that $l\nssim a$. By applying Lemma 3 to the set $\{a,b,c,d\}$
and $l$, we have $l\ssim b$. Since $l$ is adjacent to only  one  vertex  of
$\{a,k,c,d\}$,  it  follows  that $k\ssim a$,  for  otherwise  we  have  a
contradiction  to  Lemma  3.  Thus, $l\nssim a$, $l\ssim b$  and $k\ssim a$,  hence
$\langle u,a,b,h,k;v,c,d;f,g,l\rangle  \simeq  G_{17}$, a contradiction. Therefore,
$l\ssim a$.

Thus, we obtain four graphs shown in
Figure \ref{F7}(a). 
By Lemma 1, the set $\{l,f\}$ does not  dominate $F$,  so  there
exists a vertex $m\in V(F)$ such that $m\nssim \{l,f\}$.

\begin{figure}[t!] 
\vspace*{-2.5cm}
\hspace*{-1cm}
\centering\includegraphics[width=1.0\linewidth]{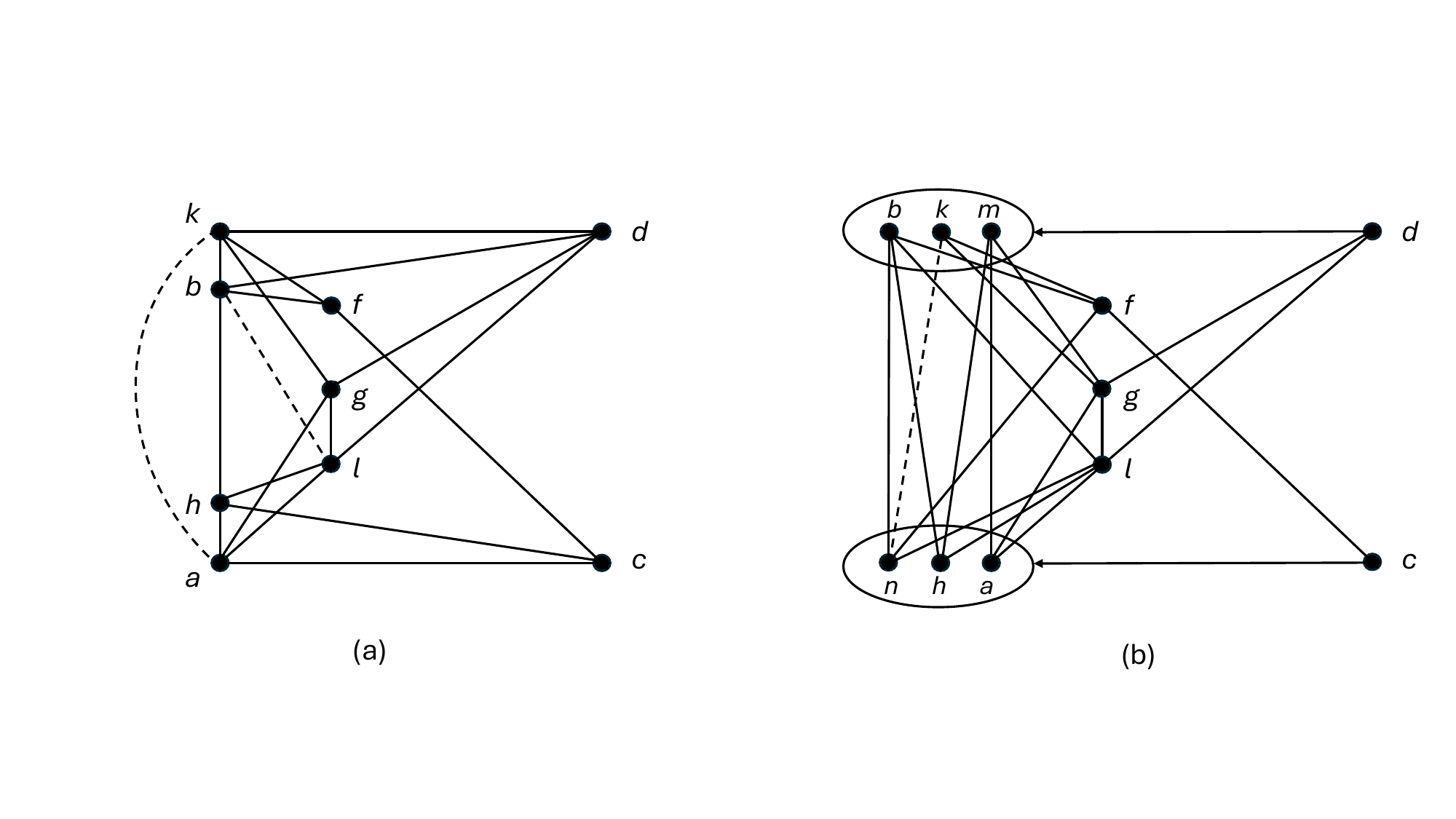} 
\vspace*{-1.6cm}
\caption{(a) Graphs used in Lemma 6; 
(b) Graphs after adding the vertex $n$.}
\label{F7} 
\vspace*{-0.2cm}
\end{figure}

\vspace*{-0.1cm}
\begin{lem}
Let $a,b,h,k\in A$, $c,d\in B$, $f,g,l\in C$   and $m\in V(F)$.
If $\langle a,h,b,k,$ $c,d,f,g,l\rangle $ is one of the graphs in Figure \ref{F7}(a)  and
$m\nssim \{f,l\}$, then

\smallskip
(1)\; $m\in A$, $m\ssim \{a,b,h,k,d,g\}$ and $m\nssim c$;
\par
(2)\; $b\ssim l$ and $a\nssim k$.
\end{lem}

\up\pf The set $\{h,k,c,d,f,l\}$ and the vertex $m$ satisfy  the
hypotheses of Lemma 5, hence there are four cases to consider.

\smallskip
{\it Case 1.} $m\in A$ and $m\ssim \{h,k,c\}$, $m\nssim d$. By  Lemma 2(iii), $m\ssim a$.
Assume that $m\nssim b$. Applying Lemma 3 to the set $\{m,b,c,d\}$ and $l,g$,
we have $b\ssim l$ and $g\ssim m$. Then, $\langle u,m,h,b,k;v,c,d;f,g,l\rangle  \simeq
G_{17}$, a contradiction. Consequently, $m\ssim b$.  Further, $b\ssim l$,  for  otherwise
$\langle h,l,c;b,d,f;m,u\rangle \simeq G_{12}$, a contradiction. 

By Lemma 1, the set $\{m,v\}$ does not dominate $F$. Hence, there
is a vertex $n\in A$ such that $n\nssim m$. Then, by Lemma 2(ii), $n\ssim d$ and,  by
Lemma 2(iii), $n\ssim b$. Now, we can apply Lemma 3  to $\{m,n,c,d\}$  and
$f,l$, which produces $n\ssim \{f,l\}$. We have $\langle l,a,v;b,m,f;d,n\rangle  \simeq  G_{10}$  or
$G_{12}$ depending on the existence of $an$, a contradiction.

\smallskip
{\it Case 2.} $m\in A$ and $m\ssim \{h,k,d\}$, $m\nssim c$. By
Lemma 2(iii), $m\ssim b$.
Assume that $m\nssim a$. Then, $\{a,m,c,d\}$ and $f$ satisfy the hypotheses  of
Lemma 3 and $f\nssim \{a,m,d\}$, contrary to the conclusion  of  Lemma  3.
Hence, $m\ssim a$. 
If 
$a\ssim k$, then $\langle k,f,d;a,c,l;m,u\rangle  \simeq  G_{12}$, so   $a\nssim k.$
If
$b\nssim l$,
then
$\langle b,f,d;h,c,l;m,u\rangle  \simeq  G_{12}$, so   $b\ssim l.$

By Lemma 1, there exists a vertex $n\in A$ such  that $n\nssim m$.  By
Lemma 2(ii), $n\ssim c$ and, by Lemma 2(iii), $n\ssim a$. Applying Lemma 3 to
$\{n,m,c,d\}$ and $l$, we have $l\ssim n$. 
Finally, if
$m\nssim g$, then $\langle l,b,v;a,m,c;g,n\rangle  \simeq  G_{7}, G_{9}, G_{10}\hbox{ or }G_{12}$, so   $m\ssim g.$
Thus, the conclusions of Lemma 6 are satisfied.

\smallskip
{\it Case 3.} $m\in B$ and $m\ssim \{c,d,h\}$, $m\nssim k$. Suppose
that $m\nssim a$. By Lemma
2(i), the  set $\{m,a\}$  dominates $\langle A\cup B\rangle $,  hence $a\ssim k$.  We  have
$\langle a,k,l;c,f,m\rangle  \simeq  G_{2}$, a contradiction, so $m\ssim a$.
Then, Lemma 2(iii) implies $m\nssim b$, since $a\nssim b$. Further, if
$a\ssim k$, then $\langle a,l,m;k,d,f\rangle  \simeq  G_{3}$, so   $a\nssim k.$
If
$b\nssim l$, then $\langle b,h,f;d,m,l\rangle  \simeq  G_{3}$, so  $b\ssim l$.
If
$m\nssim g$, then $\langle b,h,f;d,m,g\rangle  \simeq  G_{2}$, so  $m\ssim g.$

By Lemma 1, there is a vertex $n\in B$ such that $n\nssim m$.
By  Lemma  2(ii),
$\langle h,k,m,n\rangle \simeq 2K_{2}$, so $n\ssim k$ and $n\nssim h$.
By Lemma 2(iii), $n\ssim d$.
If $n\ssim a$, then we have a contradiction to Lemma 2(iii),  since $n\ssim k$
and $a\nssim k$. Hence, $n\nssim a$.
The set $\{b,m\}$ dominates $\langle A\cup B\rangle $  by  Lemma  2(i), and so $n\ssim b$. For the set $\{h,k,m,n\}$ and $l$, we can apply Lemma
3, which produces $n\ssim l$. Assume that $n\nssim g$. Then, $\{k,h,m,n,l\}$  and $g$
satisfy  the  hypotheses  of  Lemma 4.   Therefore, $g\nssim l$, a
contradiction. Consequently, $n\ssim g.$

The set $\{b,m\}$ does not dominate $F$ by Lemma 1.  Therefore,
there is a vertex $p\nssim \{b,m\}$ and $p\in C$ by Lemma 2(i). Applying Lemma
3 to $\{a,b,m,n\}$ and $p$, we obtain $p\ssim \{a,n\}$. 
If
$p\nssim g$, then $\langle u,b,p;g,n,m\rangle  \simeq  G_{3}$, so  $p\ssim g.$
If
$p\nssim l$, then $\langle v,m,p;l,a,b\rangle  \simeq  G_{3}$, so   $p\ssim l.$
Further, if
$p\nssim f$, then $\langle l,b,a;v,f,m;p,d\rangle  \simeq  G_{9}\hbox{ or }G_{12}$, so $p\ssim f.$
If
$p\ssim d$, then  $\langle l,b,a;v,f,m;p,d\rangle  \simeq  G_{13}$, so $p\nssim d.$
We now apply Lemma 3 to $\{a,b,c,d\}$ and $p$. We have $p\ssim c$.
Assume that $p\nssim h$.
The set $\{h,k,c,d,l\}$ and $p$ satisfy the hypotheses  of  Lemma
4, and so $p\nssim l$, a contradiction. Thus, $p\ssim h$. Finally,
if
$p\nssim k$, then $\langle p,f,h;g,k,m;l,v\rangle  \simeq  G_{12}$, so $p\ssim k$ and 
$\langle h,b,m;p,k,v;a,l\rangle  \simeq  G_{12}$, a contradiction.

\smallskip
{\it Case 4.} $m\in B$ and $m\ssim \{c,d,k\}$, $m\nssim h$. 
If $m\ssim \{a,b\}$,  
we obtain a contradiction to Lemma 2(iii), and
if $m\nssim \{a,b\}$, 
we obtain a contradiction to Lemma 2(i). 
If $m\nssim a$  and
$m\ssim b$, then $l\ssim b$,  for  otherwise $\langle b,f,m;h,c,l\rangle  \simeq  G_{3}$.  We  have
$\langle v,f,m;l,b,a\rangle  \simeq  G_{3}$, a contradiction. Thus, $m\ssim a$  and
$m\nssim b$.  Since
$m\ssim \{a,k\}$, by Lemma 2(iii), $k\ssim a$. 
Now,
if
$m\nssim g$, then  $\langle b,f,h;d,g,m\rangle  \simeq  G_{1}$, so $m\ssim g.$
If $b\nssim l$, then $\langle c,f,m;h,b,l\rangle  \simeq  G_{2}$, so $b\ssim l.$

By Lemma 1, there is a vertex $n\in B$ such that $n\nssim m$.  For  the
sets $\{a,b,m,n\}$ and $\{h,k,m,n\}$, we can apply Lemma  2(ii),  which
produces $n\ssim \{b,h\}$ and $n\nssim \{a,k\}$.
Then, by Lemma 2(iii), $n\ssim \{c,d\}$.
The application of Lemma 3 to the set $\{h,k,n,m\}$ and the vertices
$l$ and $f$ gives $n\ssim \{l,f\}$. 
If
$n\ssim g$, then $\langle u,f,h;g,n,m\rangle  \simeq  G_{3}$, so $n\nssim g.$

In accordance  with  Lemma  1,  the  set $\{k,n\}$  does  not
dominate $F.$
Therefore,  there  is  a  vertex $p\in V(F)$  such  that
$p\nssim \{k,n\}$. By Lemma 2(i), $p\in C$. The  application  of  Lemma  4  to
$\{h,k,n,m,f\}$ and $p$ gives $p\ssim \{h,m\}$ and $p\nssim f$.
Suppose that $l\nssim p$.  Then,
$\{h,k,n,m,f,p\}$ and $l$ satisfy the  hypotheses  of  Lemma  5,  and
hence $l\not\in C$, a contradiction. Thus, $l\ssim p$.
Analogously, $g\ssim p$. Further,
$p\nssim a$ implies $\langle h,p,a;n,d,f\rangle  \simeq G_{1}$ or $G_{2}$,
hence $p\ssim a$.  Assume  that $p\nssim d$.
The application of Lemma 3 to $\{h,k,c,d\}$ and $p$  gives $p\ssim c$,
and we have \linebreak $\langle v,d,m,c,n;u,k,h;f,l,p\rangle  \simeq G_{17}$, a  contradiction.
Hence, $p\ssim d$. 

By Lemma 1, the set $\{f,p\}$ does not dominate $F$,  and  hence
there  is  a  vertex $r\in V(F)$  such   that $r\nssim \{f,p\}$. The set
$\{h,k,n,m,f,p\}$  and $r$  satisfy  the  hypotheses  of   Lemma   5.
Therefore, there are four cases to consider.

\smallskip
\textit{Subcase 4.1.} $r\in A$ and $r\ssim \{h,k,n\}$, $r\nssim m$.
By  Lemma 2(iii),
$r\ssim b$, since $n\ssim \{b,r\}$. If $r\nssim a$,
then we obtain a contradiction to
Lemma  3,  since $f$  is
adjacent to one vertex of $\{a,r,m,n\}$. Hence, $r\ssim a$.  If $r\nssim d$,  then
$\langle d,n,p;k,f,r\rangle  \simeq  G_{3}$, a contradiction. Therefore, $r\ssim d$. Now, $r\nssim c$, for
otherwise $r\ssim \{c,d\}$ and $c\nssim d$, contrary to Lemma 2(iii). 
If
$p\nssim c$, then $\langle c,m,f;h,p,r\rangle  \simeq  G_{2}$, so $p\ssim c.$
If
$r\nssim g$, then $\langle p,c,d;u,f,r;g,a\rangle  \simeq  G_{9}$, so $r\ssim g.$
By Lemma 1, the vertex $r$ does not  dominate $\langle A\rangle $,  and  so
there is a vertex $s\in A$ such that $s\nssim r$.
By Lemma 2(ii), $s\ssim \{c,m\}$
and $s\nssim n$. Then, $s\ssim \{h,k\}$  by  Lemma  2(iii),  since $c\ssim \{h,s\}$  and
$m\ssim \{k,s\}$. 
Applying Lemma 3 to $\{r,s,n,m\}$ and $f,p$, we obtain $s\ssim \{f,p\}$. 
Finally, if
$s\ssim g$, then $\langle g,r,v;s,h,f;m,k\rangle  \simeq  G_{12}$, so $s\nssim g$ and
$\langle u,h,r,k,s;v,n,m;f,p,g\rangle  \simeq  G_{17},$ a contradiction.

\smallskip
\textit{Subcase 4.2.} $r\in A$   and $r\ssim \{h,k,m\}$, $r\nssim n$.   We   have
$\langle h,p,r;n,d,f\rangle  \simeq  G_{2}$ or $G_{3}$ depending on the existence of the  edge
$rd$, a contradiction.

\smallskip
\textit{Subcase 4.3.} $r\in B$ and $r\ssim \{n,m,h\}$, $r\nssim k$. Since $f$  is  adjacent
to only one vertex of $\{h,k,r,d\}$, it follows by Lemma 3 that $r\ssim d$.
We obtain $\langle d,p,r;k,a,f\rangle  \simeq  G_{2}$ or $G_{3}$ depending on the existence of
the edge $ra$, a contradiction.

\smallskip
\textit{Subcase 4.4.} $r\in B$ and $r\ssim \{n,m,k\}$, $r\nssim h$.
Since $k\ssim \{d,r\}$, it follows by Lemma 2(iii) that $r\ssim d$.
We  have
$\langle d,p,k;n,h,f;r,v\rangle  \simeq  G_{12}$.  This  contradiction  completes   the
proof of Lemma 6. \qed
\smallskip

Thus, by applying Lemma 6 to the graphs in Figure 
\ref{F7}(a) and the vertex $m$,
we obtain the graph in 
Figure 
\ref{F7}(b) without the vertex $n$.
By Lemma 1, since $m$ does not  dominate $\langle A\rangle $,  there  is  a
vertex $n\in A$ such that $n\nssim m$. By Lemma 2(ii), $n\ssim c$
and $n\nssim d$.  Since
$c\ssim \{n,a,h\}$, it follows by Lemma 2(iii) that $n\ssim \{a,h\}$. 
Let us apply Lemma
3 to $\{n,m,c,d\}$ and $f,l$. We have $n\ssim \{f,l\}$. Further, if
$n\nssim b$, then $\langle b,l,m;f,v,n\rangle  \simeq  G_{3}$, so $n\ssim b.$
If $n\ssim g$, then $\langle l,b,v;a,m,c;g,n\rangle  \simeq  G_{13}$, so $n\nssim g.$
The resulting graphs are shown in Figure \ref{F7}(b)
(ovals represent complete graphs).

By Lemma 1, the set $\{n,d\}$ does not dominate $F$. Hence, there
is a vertex $p\in V(F)$ such that $p\nssim \{n,d\}$.  By  Lemma 2(i), $p\in C$.
Applying Lemma 4 to $\{n,m,c,d,l\}$ and $p$,  we  obtain $p\ssim \{m,c\}$  and
$p\nssim l$. Suppose that $p\nssim f$ and apply Lemma 5 to
$\{n,m,c,d,p,l\}$ and  $f$.
We have $f\not\in C$, a contradiction. There\-fore, \mbox{$p\ssim f$.}
Assume that $p\nssim k$
and consider the
induced subgraph
$\langle n,h,m,k,c,d,f,l,p\rangle$, where
the edges $nk$ and
$ph$ are undetermined.
If $n\nssim k$,
we obtain a contradiction to Lemma 3, since $p$  is  adjacent  to  one  vertex  of
$\{n,k,c,d\}$. Applying Lemma 3 to $\{h,k,c,d\}$ and $p$, we obtain $p\ssim h$.
Now, $\langle u,k,m,h,n;v,d,c;l,f,p\rangle  \simeq G_{17}$, a contradiction. Thus,
$p\ssim k$.
\par
The set $\{l,p\}$ dominates the resulting graph. By  Lemma  1,
there is  a  vertex $r\in V(F)$  such  that $r\nssim \{l,p\}$.
Consider  the subgraph 
$\langle n,h,m,k,c,d,f,l,p\rangle$.
It is obvious that this graph and the
vertex $r$ satisfy the hypotheses of Lemma 6. By  the  conclusions
of Lemma 6, $r\in A$, $r\ssim \{k,m,h,n,c,f\}$, $r\nssim d$ and
$h\ssim p$, $k\nssim n$. 
Let us 
apply  Lemma  6
to the subgraph
$\langle n,a,m,k,c,d,f,l,p\rangle$
and the vertex $r$.
We obtain $a\ssim \{p,r\}$. Now, if
$p\nssim g$, then $\langle a,p,n;g,k,v\rangle  \simeq  G_{3}$, so $p\ssim g.$
If $r\ssim g$, then $\langle g,p,d;r,h,f\rangle  \simeq  G_{3}$, so $r\nssim g.$
Further, if 
$p\ssim b$, then $\langle b,n,d;p,a,v;k,f\rangle  \simeq  G_{12}$, so $p\nssim b.$
If $r\nssim b$, then we have a contradiction to  Lemma  3,  since $p$  is
adjacent to one vertex of $\{r,b,c,d\}$. 
To summarize, the graph $F$ contains the induced subgraph shown in Figure \ref{F7}(b), with $n\nssim k$ and $u\ssim A\cup C$ and $v\ssim B\cup C$; also, $p$ is only adjacent to 
$\{u,v,a,h,k,m,s,c,f,g\}$ and $r$ is only adjacent to $\{u,a,h,n,b,k,m,c,f\}$.

By Lemma 1, the vertex $r$ does not dominate $\langle A\rangle $, and  hence
there is a vertex $s\in A$ such that $s\nssim r$.
By Lemma 2(ii), $s\ssim d$  and
$s\nssim c$. By Lemma 2(iii), $s\ssim \{b,k,m\}$.
Applying Lemma 3 to $\{r,s,c,d\}$
and $p,g,l$, we obtain $s\ssim \{p,g,l\}$. Now, if
$s\nssim a$, then $\langle a,p,r;l,v,s\rangle  \simeq  G_{3}$, so $s\ssim a.$
If $s\nssim h$, then $\langle h,p,r;l,v,s\rangle  \simeq  G_{3}$, so $s\ssim h.$
If $s\ssim f$, then $\langle k,d,r;p,v,a;f,s\rangle  \simeq  G_{13}$, so $s\nssim f.$
The only undetermined edge is
$sn$. It is straightforward  to  see
that the graph $\langle n,r,m,s,c,d,p,l\rangle $ is isomorphic to the  graph $Q$
(defined before Lemma 6)
with the correspondence respecting the partition into the parts
$A$, $B$ and $C$. We have already proved that  the  case $a\ssim k$ 
is impossible for the graph $Q$.
Therefore, we  can  assume
that $n\nssim s$ in our case.
Thus, in the graph $F$, 
$s\ssim \{u,a,h,b,k,m,d,g,l,p\}$
and
$s\nssim \{v,n,r,c,f\}$.

By Lemma 1, the set $\{b,p\}$ does not dominate $F$, hence there
is a vertex $x$ such that $x\nssim \{b,p\}$.
If $x\in A$, then we have a contradiction to Lemma 3,  since $p$
is adjacent to one vertex  of $\{x,b,c,d\}$.  Hence, $x\not\in A$  and
$x\ssim v$.
Suppose that $x\nssim l$ and apply Lemma 6 to $\langle k,m,h,n,d,c,l,f,p\rangle $ and $x$.
We have $x\in A$, a contradiction. Therefore, $x\ssim l$. 
If $x\nssim a$, then $\langle a,r,p;l,b,x\rangle  \simeq  G_{2}\hbox{ or }G_{3}$, so $x\ssim a.$
If $x\nssim n$, then $\langle a,x,n;p,k,v\rangle  \simeq  G_{2}\hbox{ or }G_{3}$, so $x\ssim n.$
Now, if $x\ssim k$, then 
$\langle k,x,b;p,a,v\rangle  \simeq  G_{3}$, so $x\nssim k.$
Finally, if
$x\nssim r\hbox{ or }x\ssim s$, then  $\langle l,b,v;a,r,p;x,s\rangle  \simeq  G_{9},G_{12}\hbox{ or }G_{13}$, so $x\ssim r\hbox{ and }x\nssim s.$

Let $x\in B$, i.e.~$x\nssim u$. If $x\nssim h$, then $\{x,h\}$  does  not  dominate
the vertex $k$, contrary to Lemma 2(i). Hence, $x\ssim h$. Since $x\ssim n$,  we
obtain $x\nssim m$ by Lemma 2(iii).  
If $x\nssim d$, then $\langle u,r,p;l,x,d\rangle  \simeq  G_{2}$, so $x\ssim d.$
Now, if $x\nssim f$, then
$\langle b,d,f;h,x,p;r,m\rangle  \simeq  G_{13}$, so $x\ssim f$ and
$\langle b,d,h;f,x,p;n\rangle  \simeq  G_{6}$, a contradiction.

Let $x\in C$, i.e. $x\ssim u$. 
If $x\nssim c$, then $\langle l,b,v;a,m,c;x,n\rangle  \simeq  G_{12}\hbox{ or }G_{13}$, so  $x\ssim c.$
Assume that $x\nssim h$. Applying Lemma 3 to $\{h,k,c,d\}$ and $x$,  we  have
$x\ssim d$. 
Now, if $x\nssim f$, then $\langle x,c,d;u,f,h\rangle  \simeq  G_{3}$, so $x\ssim f$ and
$\langle b,d,h;f,x,p;r,k\rangle  \simeq  G_{13}$, a contradiction.
Consequently, $x\ssim h$. 
If
$x\nssim g$, then $\langle p,f,g;h,b,x\rangle  \simeq  G_{2}\hbox{ or }G_{3}$, so $x\ssim g.$
If
$x\ssim f$, then $\langle p,f,g;h,b,x;c\rangle  \simeq  G_{6}$, so $x\nssim f.$
Now, if
$x\nssim d$, then $\langle b,f,d;h,p,x\rangle  \simeq  G_{2}$, so $x\ssim d.$
If
$x\nssim m$, then $\langle m,d,p;r,x,f;h,b\rangle  \simeq  G_{13}$, so $x\ssim m.$

Thus, 
$x\ssim \{u,v,a,c,d,g,h,l,m,n,r\}$
and $x\nssim \{b,f,k,p,s\}$.
Now, the  set $\{g,r\}$  dominates
the resulting graph. By Lemma 1, there is a vertex $y\in V(F)$  such
that $y\nssim \{g,r\}$.
If $y\in A$, 
we obtain a contradiction to Lemma 3, since $g$ is  adjacent  to  one
vertex of $\{r,y,c,d\}$. Hence, $y\not\in A$ and $y\ssim v$.
\par
Let $y\in B$, i.e. $y\nssim u$. Since $y\nssim r$, it follows by  Lemma  2(i)
that $y\ssim s$. Now, $y\nssim n$ by Lemma 2(iii). Then, $y\ssim k$ by Lemma 2(i)  and
$y\nssim h$ by Lemma 2(iii). 
If $y\nssim f$, then $\langle f,k,n;v,g,y\rangle  \simeq  G_{3}$, so $y\ssim f.$
Now, if $y\nssim p$, then $\langle f,n,y;p,h,g\rangle  \simeq  G_{2}$, so $y\ssim p$ and 
$\langle h,l,r;p,g,y\rangle  \simeq  G_{2} \hbox{ or } G_{3}$, 
a contradiction.

Let $y\in C$, i.e. $y\ssim u$. Suppose that $y\nssim f$ and apply Lemma  5  to
$\{a,b,c,d,f,g\}$ and $y$. We obtain $y\not\in C$, a  contradiction.  Therefore,
$y\ssim f$. Further,
if $y\nssim p$, then $\langle p,h,g;f,r,y\rangle  \simeq  G_{2}\hbox{ or }G_{3}$, so $y\ssim p.$
If $y\nssim h$, then $\langle h,l,r;p,g,y\rangle  \simeq  G_{2}\hbox{ or }G_{3}$, so $y\ssim h.$
If $y\nssim l\hbox{ or }y\ssim m$, then $\langle p,g,f;h,l,r;y,m\rangle  \simeq  G_{9},G_{12}\hbox{ or }G_{13}$, so $y\ssim l$
and $y\nssim m.$
Finally, if
$y\nssim c$, then $\langle p,g,f;h,l,r;y,m,c\rangle  \simeq  G_{15}$, so 
$y\ssim c.$

Assume now that $y\nssim d$ and apply Lemma 3 to $\{r,s,c,d\}$ and $y$.
We obtain $y\ssim s$. Next,
if $y\nssim a$, then $\langle a,p,r;l,y,d\rangle  \simeq  G_{2}$, so $y\ssim a.$
If $y\ssim k$, then $\langle a,g,r;y,v,k;s\rangle  \simeq  G_{6}$, so $y\nssim k.$
If $y\nssim b$, then $\langle p,g,f;h,x,b;y,s\rangle  \simeq  G_{12}\hbox{ or }G_{13}$, so 
$y\ssim b.$
Further, if
$y\ssim x$, then $\langle y,x,p;b,d,r\rangle  \simeq  G_{3}$, so $y\nssim x$ and
$\langle h,b,x;p,k,v;a,y\rangle  \simeq  G_{12}$,
a contradiction. 
Therefore, 
$y\ssim d$.

Further, if
$y\nssim b$, then $\langle u,b,g;y,d,c\rangle  \simeq  G_{3}$, so $y\ssim b.$
If $y\ssim a$, then $\langle a,g,r;y,d,f;$
$h,p\rangle  \simeq  G_{12}$, so $y\nssim a.$
If $y\ssim x$, then $\langle u,g,r;y,d,c;b,p,x\rangle  \simeq  G_{16}$, so 
$y\nssim x.$
Now, if
$y\nssim s$, then $\langle d,s,x;y,h,f\rangle  \simeq  G_{3}$, so $y\ssim s.$
If $y\nssim k$, then $\langle u,y,k;x,c,d;l\rangle  \simeq  G_{6}$, so $y\ssim k.$
Finally, if
$y\nssim n$, then $\langle h,m,n;y,k,v\rangle  \simeq  G_{2}$, so $y\ssim n.$

Thus, 
$y\ssim \{u,v,b,c,d,f,h,k,l,n,p,s\}$
and $y\nssim \{a,g,m,r,x\}.$
Now, the  set $\{g,n\}$  dominates
the resulting graph. By Lemma 1, there is a vertex $z\in V(F)$  such
that $z\nssim \{g,n\}$. 
If $z\in A$, 
we obtain a contradiction to
Lemma 3, since $g$  is  adjacent
to one vertex of $\{n,z,c,d\}$. Hence, $y\not\in A$, i.e. $z\ssim v$. 
If $z\ssim a$, then $\langle a,z,n;g,v,k\rangle  \simeq  G_{2}\hbox{ or } G_{3}$, so $z\nssim a.$
If $z\nssim y$, then $\langle v,y,z;g,k,a\rangle  \simeq  G_{2}\hbox{ or }G_{3}$, so $z\ssim y.$
Finally, if
$z\nssim k$, then $\langle k,r,g;y,n,z\rangle  \simeq  G_{2}\hbox{ or }G_{3}$, so $z\ssim k$ and
$\langle k,r,g;y,n,v;s,z\rangle  \simeq  G_{7}, G_{9}, G_{10}\hbox{ or }G_{12}$, 
a contradiction.
This completes the proof of Theorem \ref{ZZ}. 
\qed

In conclusion, this proof provides a polynomial-time algorithm that transforms any dominating set 
$D$ of a domination perfect graph 
$G$ into an independent dominating set of size at most 
$\vert D\vert$.


\end{document}